\documentclass[a4paper,12pt]{amsart}
\usepackage[
hmarginratio={1:1},     
vmarginratio={1:1},     
textwidth=15.8cm,        
textheight=21cm,			
heightrounded,          
]{geometry}        

\usepackage[cp1250]{inputenc}

\usepackage{euler}
\usepackage{graphicx}
\usepackage{lpic}
\usepackage{longtable}
\usepackage{url}
\usepackage{import}
\usepackage{amsfonts, amstext, amsmath, amsthm, amscd, amssymb}
\usepackage[sc]{mathpazo}
\usepackage{enumitem, textcomp, setspace}
\usepackage{caption}

\numberwithin{equation}{section}

\theoremstyle{plain}
\newtheorem{theorem}{Theorem}[section]
\newtheorem{proposition}[theorem]{Proposition}
\newtheorem{lemma}[theorem]{Lemma}

\theoremstyle{definition}
\newtheorem{definition}[theorem]{Definition}

\newtheorem{remark}[theorem]{Remark}

\usepackage[pagebackref=false]{hyperref}

\usepackage[all]{hypcap}

\definecolor{newblue}{rgb}{0.27, 0.32, 0.86}
\definecolor{newred}{rgb}{0.86, 0.32, 0.27}
\hypersetup{
	pagebackref=true,
	colorlinks=true,       
	linkcolor=newred,          
	citecolor=newblue,        
	filecolor=magenta,      
	urlcolor=newblue           
}

\usepackage[]{appendix}
\usepackage{wrapfig}

\providecommand{\subjclass}[1]{\textbf{\textit{2020 Mathematics Subject Classification:}} #1}

\graphicspath{ {fig/} } 

\title{A resistance invariant of special alternating links}
\author{Micha\l \;Jab\l onowski}
\address{Institute of Mathematics, Faculty of Mathematics, Physics and Informatics,\newline University of Gda\'nsk, 80-308 Gda\'nsk, Poland}

\email{michal.jablonowski@ug.edu.pl}

\date{\today}

\begin{document}


\maketitle

\begin{abstract}
We introduce a Laplacian trace invariant for special, reduced,
alternating diagrams of oriented knots and links and show that it is
determined by two consecutive coefficients at an extreme of the
Alexander polynomial.
The invariant also admits an interpretation as one half of the sum
of the directed resistances associated with the normalized balanced
Tait-graph Laplacian. Explicit flype-related examples show that the
Laplacian characteristic polynomial need not be preserved although the invariant is preserved.
\end{abstract}

\section{Introduction}

Alternating link diagrams occupy a central position in knot theory because their combinatorics is unusually rigid: diagrammatic data often controls classical invariants, and local moves can be globally constrained. In particular, the flyping theorem provides a strong form of
diagrammatic rigidity for prime alternating links: any two reduced alternating diagrams of the same oriented prime link are related by a sequence of flypes.
\par 
Our construction lives in the special alternating setting. Briefly, a special diagram is one in which the Seifert circles from Seifert’s algorithm bound disks that are not nested; this hypothesis is standard when one wants a clean interface between an oriented diagram and a checkerboard spanning surface.
\par 
Given a checkerboard shading of a connected link diagram, one obtains associated checkerboard graphs (often called Tait graphs) by placing vertices in regions and edges at crossings. These graphs and their planar duality encode substantial information, and they are a familiar bridge between knot theory and graph polynomials.
\par 
Our invariant $FP$ is not ``just spectral''. A flype changes the local placement of crossings inside a tangle; on the graph side, it alters Tait graphs by a local transformation that need not preserve the Laplacian up to conjugation. Consequently, the Laplacian characteristic polynomial (and hence the multiset of eigenvalues) can genuinely vary between flype-related diagrams. Our invariant gives a resistance interpretation of an Alexander-coefficient invariant for special alternating links.

\section{Definitions}\label{sDefinitions}

A \emph{diagram} $D$ of a (non-trivial) knot or link $KL$ is a generic projection of the knot or link in the plane, a regular $4$-valent graph, with extra information at each vertex indicating which arc of the link passes over the other. Moreover, it is \emph{alternating} when traveling around each link component on a diagram, we pass crossing alternately over and under (i.e., if one passes an over/under-crossing, the next crossing will be an under/over-crossing, see \cite{Men21} for a recent survey).
\par 
Throughout the paper, all diagrams are assumed to have connected underlying projection graphs. A diagram is called \emph{reduced} if it has no nugatory crossings., i.e. , a crossing whose neighborhood looks like the one in Figure \ref{rys15_1}, that can be removed without changing the knot type.

\begin{figure}[h!t]
	\begin{center}
		\includegraphics[width=4cm]{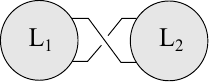}
		\caption{A nugatory crossing.\label{rys15_1}}
	\end{center}
\end{figure}

By a \emph{region} of a reduced diagram $D$ of a link $KL$ we mean a face of the underlying projection of the link. The diagram can be checkerboard-shaded so that every edge separates a shaded region (black region) from an unshaded one (white region).
\par
A standard operation of smoothing crossings of an oriented $D$, results in a collection of Seifert disks, possibly nested, in the plane that can be rejoined by half-twisted bands of $D$ to produce a Seifert surface for a link $KL$. The diagram $D$ is called \emph{special} if the Seifert disks are not nested, see \cite{SilWil19}. Any diagram $D$ can be deformed to a special diagram \cite{BurZie03}. From now on, we will always assume this condition on a knot or link diagram and that the shaded regions correspond to the Seifert surface.

\begin{figure}[h!t]
	\begin{center}
		\includegraphics[width=8cm]{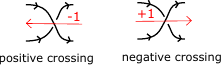}
		\caption{The convention for positive and negative crossings.\label{rys15_6}}
	\end{center}
\end{figure}

We consider a graph $\Gamma(D)$ associated with an oriented diagram $D$ as the Tait graph for unshaded checkerboard regions of $D$. The graph $\Gamma(D)$ has vertices corresponding to the unshaded regions of $D$, and an oriented edge between two vertices for every crossing at which the corresponding regions meet.
\par
On each edge $e$ of the graph $\Gamma(D)$, we assign its weight $wg(e)$ as $-1$ when the corresponding crossing is a positive, and $+1$ if it is a negative crossing. See Figure \ref{rys15_6} for the convention of the orientation of an edge and its weight (marked in red).

\par
The Laplacian matrix $L(\Gamma)$ of a directed graph $\Gamma$ is the square matrix indexed by the vertex set of $\Gamma$ with diagonal entries equal to the weighted out-degree of the vertex $v_i$, i.e., the sum of weights of edges with initial vertex $v_i$, and non-diagonal entries $L(\Gamma)_{i,j}$ equal to $-1$ times the sum of the weights of edges from $v_i$ to $v_j$.
\par
The Alexander polynomial of a knot or link $KL$ can be calculated
from the Laplacian matrix associated with a special diagram as
follows. Let $v_{\infty}$ denote the vertex corresponding to the
unbounded unshaded region. By \cite{SilWil19}, the principal
submatrix $S$ obtained from $L$ by deleting the row and column
corresponding to $v_{\infty}$ is a Seifert matrix for $KL$.
Consequently,
$
P_D(t)=\det(S-tS^T)
$
is a polynomial representative of the one-variable Alexander
polynomial, well defined up to multiplication by $\pm t^k$.

For the directed Tait graph of a special diagram, the weighted
in-degree and weighted out-degree agree at every vertex
\cite{SilWil19}. Hence
$
L\mathbf{1}=0,
\;
L^T\mathbf{1}=0.
$
For a special alternating diagram all crossings have the same sign,
and hence all edge weights are equal to a common value
$\omega\in\{-1,+1\}$; see \cite{Sto05}.
After multiplying $L$ by this common sign, one obtains the
Laplacian of a connected balanced directed graph. A connected balanced
directed graph is strongly connected, and hence its Laplacian has
rank $n-1$. Therefore, 
$
\operatorname{rank}(L)=n-1.
$

Let us define the new ``resistance'' invariant.

\begin{definition}
	For a special, reduced, alternating diagram $D$, define 
	$$FP(D)=\operatorname{tr}(L^TL^+),$$ where $L=L(\Gamma(D))$ and $L^+$ is the Moore–Penrose pseudoinverse of $L$. 
\end{definition}

We have the main theorem of this paper, as follows.

\begin{theorem}\label{thm:main}
	For connected, special, reduced, alternating diagrams $D_1$, $D_2$ of the same oriented knot or link $KL$, we have $FP(D_1)=FP(D_2)$.
\end{theorem}

\begin{remark}
	The name ``resistance'' for the invariant $FP$ comes from the
	following interpretation. If $\omega\in\{-1,+1\}$ is the common
	edge weight of the directed Tait graph and
	$
	\widehat L=\omega L,
	$
	then $\widehat L$ is the corresponding Laplacian with positive unit
	edge weights. If $\rho(e)$ denotes the resistance of a directed edge
	defined from $\widehat L^+$ in the sense of \cite{BBG23}, then
	$
	FP(D)=\frac{1}{2}\sum_e \rho(e).
	$
	Equivalently, if $r(e)$ is defined directly from the signed
	pseudoinverse $L^+$, then
	$
	FP(D)=\frac{\omega}{2}\sum_e r(e).
	$
	This relation will be proved in Section~\ref{sProof}.
	
	If the transpose is omitted from the definition of $FP$, then the
	expression trivializes, since
	$
	\operatorname{tr}(LL^+)=\operatorname{rank}(L).
	$
\end{remark}

\section{Examples}\label{sExamples}

	\begin{figure}[h!t]
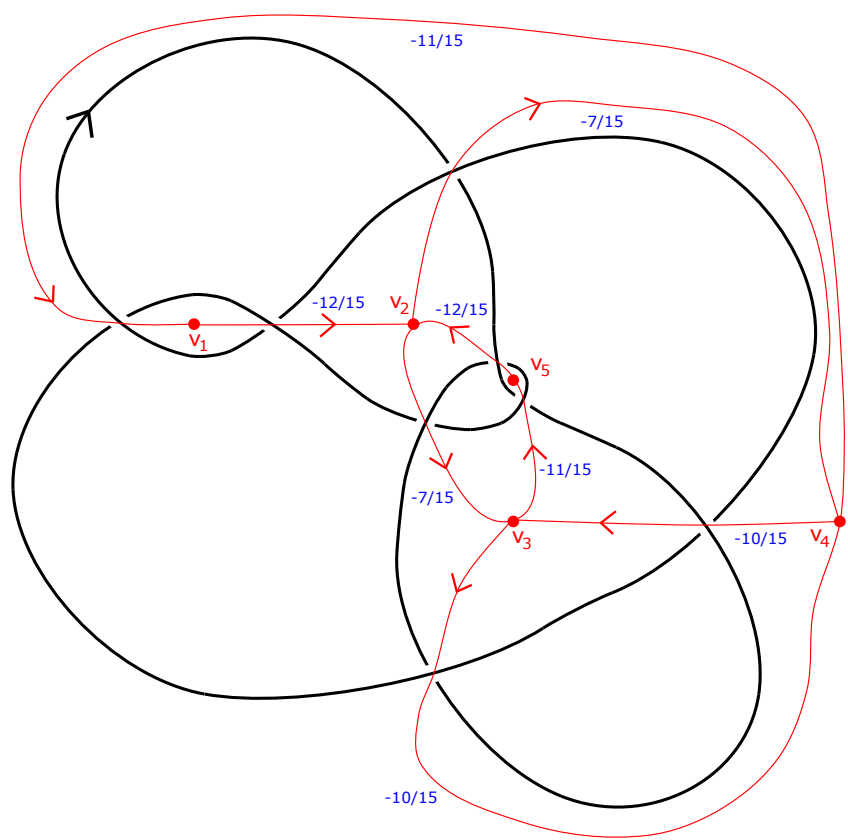

	\begin{center}
		\begin{lpic}[]{K8a2a(12cm)}
		\end{lpic}
		\caption{The diagram $8a2A$.\label{K8a2a}}
	\end{center}
\end{figure}

\begin{figure}[h!t]
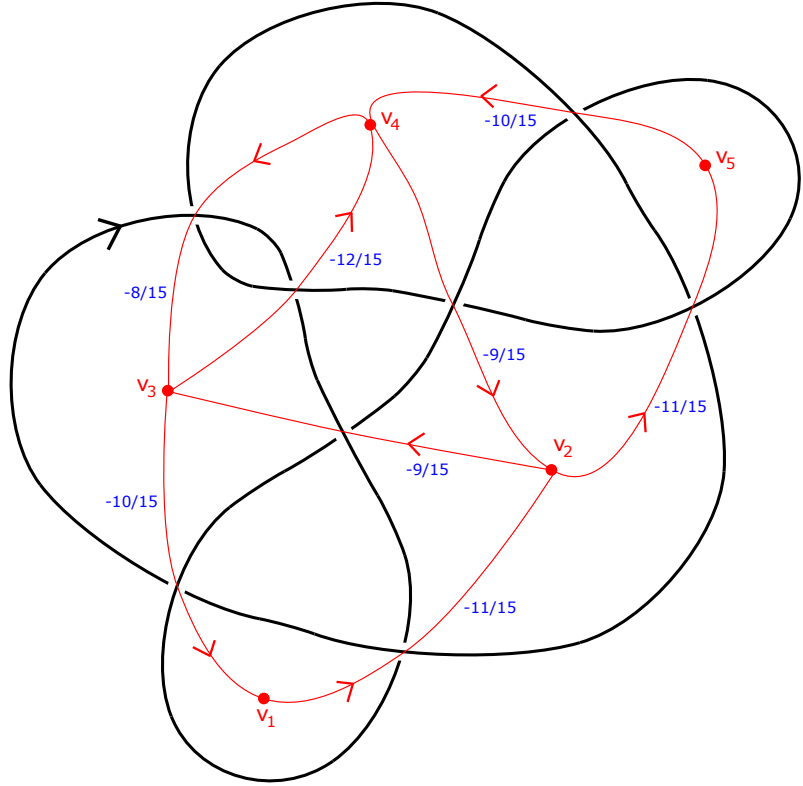

	\begin{center}
		\begin{lpic}[]{K8a2b(12cm)}
		\end{lpic}
		\caption{The diagram $8a2B$.\label{K8a2b}}
	\end{center}
\end{figure}

In Figures~\ref{K8a2a}--\ref{K8a2b} we show diagrams $8a2A$ and
$8a2B$ of the knot $8a2$, together with their oriented Tait graphs
(in red) and the directed resistances $\rho(e)$ associated with the
normalized Laplacian $\widehat L$ (in blue). The Laplacians for the graphs are as follows.

$
L(\Gamma(8a2A))=
\begin{pmatrix}
	-1 & 1 & 0 & 0 & 0 \\
	0 & -2 & 1 & 1 & 0 \\
	0 & 0 & -2 & 1 & 1 \\
	1 & 0 & 1 & -2 & 0 \\
	0 & 1 & 0 & 0 & -1
\end{pmatrix}
$
$
L(\Gamma(8a2B))=
\begin{pmatrix}
	-1 & 1 & 0 & 0 & 0 \\
	0 & -2 & 1 & 0 & 1 \\
	1 & 0 & -2 & 1 & 0 \\
	0 & 1 & 1 & -2 & 0 \\
	0 & 0 & 0 & 1 & -1
\end{pmatrix}
$

The characteristic polynomials of the Laplacians are as follows. We use the convention
$
\chi_L(\lambda)=\det(L-\lambda I).
$
$$\chi(L(\Gamma(8a2A)))=-\lambda^5-8\lambda^4-24\lambda^3-32\lambda^2-15\lambda$$
$$\chi(L(\Gamma(8a2B)))=-\lambda^5-8\lambda^4-24\lambda^3-31\lambda^2-15\lambda$$

We see that the polynomials are not the same (their coefficients of $\lambda^2$ differ by $1$), so the set of eigenvalues for $L(\Gamma(D))$ alone cannot be an invariant of the link with a special alternating diagram $D$. Let us calculate the invariant $FP$ for these diagrams.
\\ 
\par 
\noindent
$\displaystyle FP(8a2A)=\frac{1}{75}\operatorname{tr}(\begin{pmatrix}
	-1 & 1 & 0 & 0 & 0 \\
	0 & -2 & 1 & 1 & 0 \\
	0 & 0 & -2 & 1 & 1 \\
	1 & 0 & 1 & -2 & 0 \\
	0 & 1 & 0 & 0 & -1
\end{pmatrix}^T\begin{pmatrix}
	-48 & -3 & 12 & 12 & 27 \\
	12 & -18 & -3 & -3 & 12 \\
	17 & 12 & -23 & 2 & -8 \\
	-8 & 12 & 2 & -23 & 17 \\
	27 & -3 & 12 & 12 & -48
\end{pmatrix})=\frac{8}{3}$
\\ 
\ 
\\ 
\noindent
$\displaystyle FP(8a2B)=\frac{1}{75}\operatorname{tr}(\begin{pmatrix}
	-1 & 1 & 0 & 0 & 0 \\
	0 & -2 & 1 & 0 & 1 \\
	1 & 0 & -2 & 1 & 0 \\
	0 & 1 & 1 & -2 & 0 \\
	0 & 0 & 0 & 1 & -1
\end{pmatrix}^T\begin{pmatrix}
	-44 & -4 & 16 & 21 & 11 \\
	16 & -19 & 1 & 6 & -4 \\
	-9 & 6 & -24 & 6 & 21 \\
	11 & 1 & -4 & -24 & 16 \\
	26 & 16 & 11 & -9 & -44
\end{pmatrix})=\frac{8}{3}$
\\

\section{Proof of the main theorem and the resistance interpretation}
\label{sProof}

We first express the invariant $FP$ directly in terms of a Seifert
matrix. This will give a proof of Theorem~\ref{thm:main} which does
not require an analysis of flype moves.

Let $D$ be a special, reduced, alternating diagram, let
$\Gamma=\Gamma(D)$ be its oriented Tait graph, and let
$L=L(\Gamma)$ be the corresponding $n\times n$ Laplacian matrix.
Let $\mathbf{1}_k$ denote the column vector in $\mathbb{R}^k$ all of
whose entries are equal to $1$.

By the definition of the Laplacian,
$
L\mathbf{1}_n=0.
$
Moreover, the directed Tait graph occurring in the present
construction is balanced, so that
$
L^T\mathbf{1}_n=0.
$
We also use the fact recalled in Section~2 that
$
\operatorname{rank}(L)=n-1.
$

Let $v_{\infty}$ be the vertex corresponding to the unbounded
unshaded region. Relabel the vertices so that $v_n=v_{\infty}$, and
let $S$ be the $(n-1)\times(n-1)$ principal submatrix obtained from
$L$ by deleting its last row and last column. By
\cite[Theorem~1.1(i)]{SilWil19}, $S$ is a Seifert matrix associated
with $D$.

\begin{lemma}
	\label{lem:factorization}
	The matrix $S$ is nonsingular. Moreover, if
	$
	C=
	\begin{pmatrix}
		I_{n-1}\\
		-\mathbf{1}_{n-1}^T
	\end{pmatrix},
	$
	then
	$
	L=CSC^T.
	$
\end{lemma}

\begin{proof}
	Since
	$
	L\mathbf{1}_n=0
	\;\text{and}\;
	L^T\mathbf{1}_n=0,
	$
	and $\operatorname{rank}(L)=n-1$, we have
	$
	\ker L=\ker L^T=\operatorname{span}\{\mathbf{1}_n\}.
	$
	Because $\operatorname{rank}(L)=n-1$, the adjugate
	$\operatorname{adj}(L)$ is nonzero. From
	$
	L\operatorname{adj}(L)
	=
	\operatorname{adj}(L)L
	=
	0
	$
	it follows that every column and every row of
	$\operatorname{adj}(L)$ is a multiple of $\mathbf{1}_n$. Hence
	$
	\operatorname{adj}(L)=c\,\mathbf{1}_n\mathbf{1}_n^T
	$
	for some $c\neq0$. In particular, every principal cofactor of $L$
	is nonzero. Therefore
	$
	\det S\neq0.
	$
	
	Now write $L$ in block form
	$
	L=
	\begin{pmatrix}
		S & u\\
		v^T & \alpha
	\end{pmatrix}.
	$
	The equalities $L\mathbf{1}_n=0$ and
	$L^T\mathbf{1}_n=0$ imply
	$
	u=-S\mathbf{1}_{n-1},
	\;
	v^T=-\mathbf{1}_{n-1}^T S,
	$
	and consequently
	$
	\alpha=\mathbf{1}_{n-1}^T
	S\mathbf{1}_{n-1}.
	$
	Thus
	$
	L=
	\begin{pmatrix}
		S & -S\mathbf{1}_{n-1}\\
		-\mathbf{1}_{n-1}^TS &
		\mathbf{1}_{n-1}^TS\mathbf{1}_{n-1}
	\end{pmatrix}
	=
	CSC^T.
	$
\end{proof}

\begin{lemma}
	\label{lem:pseudoinverse}
	With the notation above,
	$
	FP(D)=\operatorname{tr}(S^T S^{-1}).
	$
\end{lemma}

\begin{proof}
	Put
	$
	H=C^TC.
	$
	Since $C$ has full column rank, $H$ is invertible. Define
	$
	X=
	CH^{-1}S^{-1}H^{-1}C^T.
	$
	Using $L=CSC^T$ and $C^TC=H$, we obtain
	$$
	LX
	=
	CSC^TCH^{-1}S^{-1}H^{-1}C^T
	=
	CH^{-1}C^T,
	$$
	and similarly
	$
	XL=CH^{-1}C^T.
	$
	The matrix
	$
	P=CH^{-1}C^T
	$
	is the orthogonal projection onto the column space of $C$, which is
	$\mathbf{1}_n^\perp$. Hence, $P=P^T$, and
	$
	PL=L,
	\;
	LP=L.
	$
	It follows that
	$
	LXL=L,
	\;
	XLX=X,
	$
	while both $LX$ and $XL$ are symmetric. Thus, $X$ satisfies the four
	Moore--Penrose equations, and therefore
	$
	L^+=CH^{-1}S^{-1}H^{-1}C^T.
	$
	
	Consequently,
	\begin{align*}
		FP(D)
		&=\operatorname{tr}(L^T L^+)\\
		&=\operatorname{tr}
		\left(
		CS^TC^T
		CH^{-1}S^{-1}H^{-1}C^T
		\right)\\
		&=\operatorname{tr}
		\left(
		CS^TS^{-1}H^{-1}C^T
		\right)\\
		&=\operatorname{tr}
		\left(
		S^TS^{-1}H^{-1}C^TC
		\right)\\
		&=\operatorname{tr}(S^TS^{-1}),
	\end{align*}
	as claimed.
\end{proof}

We can now relate $FP(D)$ directly to the Alexander polynomial.
Set
$
P_D(t)=\det(S-tS^T).
$
By the Seifert-matrix construction recalled in Section~2,
$P_D(t)$ is a polynomial representative of the Alexander polynomial
of the oriented link represented by $D$; see \cite{SilWil19}.

\begin{proposition}
	\label{prop:alexanderFP}
	For every connected, special, reduced, alternating diagram $D$,
	$
	FP(D)
	=
	-\frac{P_D'(0)}{P_D(0)}.
	$
	In particular, if
	$
	P_D(t)=a_0+a_1t+\cdots+a_mt^m,
	\; a_0\neq0,
	$
	then
	$
	FP(D)=-\frac{a_1}{a_0}.
	$
\end{proposition}

\begin{proof}
	Since $S$ is nonsingular,
	$
	P_D(0)=\det S\neq0.
	$
	Jacobi's formula for the derivative of a determinant gives
	\begin{align*}
		P_D'(0)
		&=
		\det(S)\,
		\operatorname{tr}
		\left(
		S^{-1}(-S^T)
		\right)\\
		&=
		-\det(S)\,
		\operatorname{tr}(S^{-1}S^T).
	\end{align*}
	By the cyclicity of the trace,
	$
	\operatorname{tr}(S^{-1}S^T)
	=
	\operatorname{tr}(S^TS^{-1}),
	$
	and Lemma~\ref{lem:pseudoinverse} therefore gives
	$
	P_D'(0)
	=
	-P_D(0)\,FP(D).
	$
	Hence
	$
	FP(D)
	=
	-\frac{P_D'(0)}{P_D(0)}.
	$
	The second assertion follows immediately from
	$
	P_D(0)=a_0,
	\;
	P_D'(0)=a_1.
	$
\end{proof}

\begin{proof}[Proof of Theorem~\ref{thm:main}]
	Let $D_1$ and $D_2$ be connected, special, reduced, alternating diagrams of the
	same oriented knot or link, and let
	$
	P_{D_i}(t)=\det(S_i-tS_i^T),
	\; i=1,2,
	$
	where $S_i$ is the principal submatrix obtained by deleting the row and
	column corresponding to the unbounded unshaded region.
	
	Since $P_{D_1}$ and $P_{D_2}$ represent the Alexander polynomial of
	the same oriented link, there is an integer $k$ such that
	$
	P_{D_2}(t)=\pm t^k P_{D_1}(t).
	$
	By Lemma~\ref{lem:factorization},
	$
	P_{D_i}(0)=\det S_i\neq0
	\; (i=1,2).
	$
	Both $P_{D_1}(t)$ and $P_{D_2}(t)$ are ordinary polynomials with
	nonzero constant term. Therefore, the equality above is possible
	only for $k=0$. Hence
	$
	P_{D_2}(t)=\pm P_{D_1}(t).
	$
	The sign disappears in the logarithmic derivative, and
	Proposition~\ref{prop:alexanderFP} gives
	$
	FP(D_2)
	=
	-\frac{P_{D_2}'(0)}{P_{D_2}(0)}
	=
	-\frac{P_{D_1}'(0)}{P_{D_1}(0)}
	=
	FP(D_1).
	$
	This proves the theorem.
\end{proof}

We finish by explaining the resistance interpretation of $FP$.

For a special alternating diagram, all crossings have the same sign,
and hence all edge weights of $\Gamma(D)$ are equal to a common
constant
$
\omega\in\{-1,+1\};
$
see \cite{Sto05}. 
Define the normalized Laplacian
$
\widehat L=\omega L.
$
If $m_{ij}>0$ directed edges have initial vertex $v_i$ and
terminal vertex $v_j$, replace them by a single directed edge
with scalar weight
$
W_{ij}=\frac{1}{m_{ij}}.
$
With the convention of \cite{BBG23}, the corresponding
off-diagonal Laplacian entry is
$
-W_{ij}^{-1}=-m_{ij}.
$
Thus the resulting weighted directed graph has Laplacian
$\widehat L$.

The underlying undirected graph is connected. Moreover,
$
\widehat L\mathbf{1}_n=\widehat L^T\mathbf{1}_n=0,
$
so the resulting weighted directed graph is balanced. A connected
balanced directed graph is strongly connected. Thus the weighted
resistance theory of \cite{BBG23} applies to $\widehat L$.
Since $\omega^2=1$,
$
\widehat L^+=\omega L^+
$
and consequently
$
\operatorname{tr}(\widehat L^T\widehat L^+)
=
\operatorname{tr}(L^TL^+)
=
FP(D).
$

For vertices $v_i,v_j$, define
$
\rho_{ij}
=
(\widehat L^+)_{ii}
+
(\widehat L^+)_{jj}
-
2(\widehat L^+)_{ij}.
$
This is the resistance associated with the balanced directed
Laplacian $\widehat L$ in the sense of \cite{BBG23}. Let
$\mathcal R=[\rho_{ij}]$ and put
$
d=\operatorname{diag}(\widehat L^+).
$
Then
$
\mathcal R
=
d\mathbf{1}_n^T
+
\mathbf{1}_n d^T
-
2\widehat L^+.
$
Since
$
\widehat L\mathbf{1}_n=0,
\;
\widehat L^T\mathbf{1}_n=0,
$
we obtain
\begin{align*}
	\operatorname{tr}(\widehat L^T\mathcal R)
	&=
	-2\operatorname{tr}
	(\widehat L^T\widehat L^+)
	=
	-2FP(D).
\end{align*}

Let $m_{ij}$ denote the number of directed edges from $v_i$ to
$v_j$. For $i\neq j$,
$
\widehat L_{ij}=-m_{ij},
$
and $\rho_{ii}=0$. Therefore
\begin{align*}
	\operatorname{tr}(\widehat L^T\mathcal R)
	&=
	\sum_{i,j}\widehat L_{ij}\rho_{ij}
	=
	-\sum_{i\neq j}m_{ij}\rho_{ij}
	=
	-\sum_e \rho(e),
\end{align*}
where the last sum is taken over all directed edges of $\Gamma(D)$
and $\rho(e)$ denotes the resistance between the initial and terminal
vertices of $e$. Comparing the two expressions gives
$
	FP(D)=\frac{1}{2}\sum_e \rho(e).
$

Equivalently, if one defines the signed quantities
$
r_{ij}
=
(L^+)_{ii}
+
(L^+)_{jj}
-
2(L^+)_{ij},
$
then
$
\rho_{ij}=\omega r_{ij},
$
and hence
$
FP(D)=\frac{\omega}{2}\sum_e r(e).
$


\end{document}